\documentclass[12pt]{amsart}
\usepackage{amssymb,hyperref}

\newtheorem{theorem}{Theorem}[section]
\newtheorem{definition}[theorem]{Definition}
\newtheorem{proposition}[theorem]{Proposition}

\newtheorem{problem}[theorem]{Problem}

\begin{document}

\author{Teodor Banica}
\address{T.B.: Department of Mathematics, Cergy-Pontoise University, 95000 Cergy-Pontoise, France. {\tt teodor.banica@u-cergy.fr}}

\author{Adam Skalski}
\address{A.S.: Institute of Mathematics of the Polish Academy of Sciences, ul. \'Sniadeckich 8, 00-956 Warszawa, Poland. {\tt a.skalski@impan.pl}}

\author{Piotr So\l tan}
\address{P.S.: Department of Mathematical Methods in Physics, Faculty of Physics, Warsaw University, Poland. {\tt piotr.soltan@fuw.edu.pl}}

\title[Noncommutative homogeneous spaces]
{Noncommutative homogeneous spaces: the matrix case}


\begin{abstract}
Given a quantum subgroup $G\subset U_n$ and a number $k\leq n$ we can form the homogeneous space $X=G/(G\cap U_k)$, and it follows from the Stone-Weierstrass theorem that $C(X)$ is the algebra generated by the last $n-k$ rows of coordinates on $G$. In the quantum group case the analogue of this basic result doesn't necessarily hold, and we discuss here its validity, notably with a complete answer in the group dual case. We focus then on the ``easy quantum group'' case, with the construction and study of several algebras associated to the noncommutative spaces of type $X=G/(G\cap U_k^+)$.
\end{abstract}

\maketitle

\section*{Introduction}

The notion of ``noncommutative space'' goes back to an old theorem of Gelfand, stating that any commutative unital $C^*$-algebra must be of the form $C(X)$, with $X$ compact Hausdorff space. In view of this fundamental result, a ``noncommutative compact space'' should be just the abstract categorical dual of an arbitrary unital $C^*$-algebra.

The interesting examples of such spaces abound. Connes proposed in \cite{co2} an axiomatic framework for the noncommutative (spin) manifolds. This notion covers a huge number of interesting situations, for instance the Standard Model one. See \cite{cco}, \cite{co1}.

Another class of interesting examples is provided by the compact quantum groups, axiomatized by Woronowicz in \cite{wo1}, \cite{wo2}, \cite{wo3}. This class covers for instance all the duals of discrete groups, as well as the compact forms of the Lie-type algebras of Drinfeld \cite{dri} and Jimbo \cite{jim}, at $q\in\mathbb R$. Note that by \cite{ntu} these latter quantum groups have a Dirac operator in the sense of Connes. The Drinfeld-Jimbo quantum groups at $|q|=1$ are not very far either, because they are known to correspond to tensor $C^*$-categories \cite{wen}.

Yet another class of interesting examples are provided by the noncommutative homogeneous spaces. For some pioneering axiomatization work and for various results in this direction, see Podle\'s \cite{pod} and Boca \cite{boc}, Tomatsu \cite{tom}, Vaes \cite{vae}, Kasprzak \cite{kas}. A number of key examples, mostly related to the various noncommutative spheres, were systematically investigated in a series of papers by Connes, D'Andrea, Dabrowski, Dubois-Violette, Khalkhali, Landi, van Suijlekom and Wagner \cite{cdu}, \cite{cla}, \cite{dd+}, \cite{kls}.

The starting point for the present paper is the following result:

\medskip
\noindent {\bf Theorem.} {\em Let $G\subset U_n$ be a closed subgroup, let $k\leq n$, and set $H=G\cap U_k$, where the embedding $U_k\subset U_n$ is given by $g\to diag(g,1_{n-k})$. Then $C(G/H)$, viewed as subalgebra of $C(G)$, is generated by the coordinate functions $u_{ij}(g)=g_{ij}$, with $i>k,j>0$.}
\medskip

This result follows indeed from the Stone-Weierstrass theorem, because the subalgebra $C_\times(G/H)\subset C(G/H)$ generated by $\{u_{ij}|i>k,j>0\}$ separates the points of $G/H$.

As a basic application, in the case $G=O_n$ and $k=n-1$ the above result tells us, modulo a standard symmetry argument, that each row of coordinates $\{u_{i1},\ldots,u_{in}\}$ on $O_n$ can be identified with the standard coordinates $\{x_1,\ldots,x_n\}$ on the sphere $S^{n-1}$.

In the quantum group case now, an analogue of the above statement can be formulated, in terms of Wang's quantum group $U_n^+$ \cite{wa1}. More precisely, if $G\subset U_n^+$ is a quantum subgroup and $k\leq n$, we can construct the quantum group $H=G\cap U_k^+$, and then the algebra $C(G/H)\subset C(G)$. Also, if we denote by $C_\times (G/H)\subset C(G)$ the $C^*$-algebra generated by the coordinates $\{u_{ij}|i>k,j>0\}$, then we have an inclusion $C_\times(G/H)\subset C(G/H)$, and the problem is whether this inclusion is an isomorphism or not.

Our first result is a complete answer to this question in the group dual case:

\medskip
\noindent {\bf Theorem A.} {\em Assume that $G=\widehat{\Gamma}$ is a discrete group dual, with $\Gamma=<g_1,\ldots,g_n>$, and with the embedding $\widehat{\Gamma}\subset U_n^+$ given by $u\to JDJ^*$, where $D=diag(g_i)$ and $J\in U_n$. Then $C_\times(G/H)\subset C(G/H)$ is an isomorphism iff $\Lambda\triangleleft\Gamma$, where $\Lambda=<g_r|\exists\,i>k, J_{ir}\neq 0>$.}
\medskip

This result is actually a consequence of a slightly more precise statement, saying that the inclusion $C_\times(G/H)\subset C(G/H)$ coincides with the inclusion $C^*(\Lambda)\subset C^*(\Lambda')$, where $\Lambda'\subset\Gamma$ is the normal closure of $\Lambda$. Observe also that in the case of diagonal embeddings ($J=1_n$) the normality condition in the statement is simply $<g_{k+1},\ldots,g_n>\triangleleft<g_1,\ldots,g_n>$.

With this theoretical result in hand, we will focus then on the case of ``easy quantum groups'' $G=(G_n)$, axiomatized in \cite{bsp}, and studied in \cite{ez1}, \cite{ez2}, \cite{ez3}, \cite{bsk}, \cite{csp}. The point is that the ``row algebras'' of type $C_\times(G/H)$ are of particular interest in the easy case:
\begin{enumerate}
\item First, $H=G_n\cap U_k^+$ is a quantum subgroup of $G_k$, and the inclusion $H\subset G_k$ itself is a quite interesting object: we will prove here, as part of Theorem B below, that $H=G_k$ is equivalent to a certain key combinatorial condition, introduced in \cite{ez2}.

\item Second, even in classical case, the quotient spaces of type $G_n/G_k$ are quite subtle. For instance one strategy in the exact computation of polynomial integrals over $O_n$ is that of examining the spaces $O_n/O_k$, with $r=n-k$ increasing. See \cite{bsc}.

\item At the quantum level now, the simplest case is $G=O_n^+$ and $k=n-1$, corresponding to the ``free spheres'' studied in \cite{bgo}. Also, the ``free hypergeometric laws'' in \cite{bbc} seem to come from spaces of type $S_r^+\backslash S_n^+/S_k^+$, not axiomatized yet.
\end{enumerate}

Summarizing, the problematics in the easy case, be it classical or free, is related to a number of recent computations and considerations. Our second result in this paper, dealing with the inclusions $C_\times(G/H)\subset C(G/H)$ in the easy case, is as follows:

\medskip
\noindent {\bf Theorem B.} {\em For an easy quantum group $G=(G_n)$ we have $G_k=G_n\cap U_k^+$ for any $k\leq n$ if and only if the category of partitions for $G$ is stable by removing blocks. If so is the case, and if $G$ is free (i.e. $S^+\subset G$), then $C_\times(G_n/G_k)\subset C(G_n/G_k)$ is in general proper.}
\medskip

We refer to sections 3-4 below for the precise statement here. The proof of the first assertion is purely combinatorial, using the results in \cite{bsp}, and the proof of the second assertion is inspired by Theorem A, by using a suitable group dual subgroup of $S_n^+$.

The above result raises the question of locating $C_\times(G/H)$ inside $C(G/H)$. We do not have results here, but as a further step in investigating $C_\times(G/H)$, we will construct and study a certain universal algebra $C_+(G/H)$, having $C_\times(G/H)$ as quotient.

Our result here is a bit technical. Let us recall from \cite{bsp} that in the free case there are exactly 6 easy quantum groups, namely $O_n^+,S_n^+,H_n^+,B_n^+$, which satisfy the combinatorial condition in Theorem B, plus two more quantum groups $S_n'^+,B_n'^+$, which do not satisfy it. So, the 4 quantum groups that we are interested in are the orthogonal, symmetric, hyperoctahedral and bistochastic groups $O_n^+,S_n^+,H_n^+,B_n^+$. These are defined by the fact that their fundamental corepresentation is orthogonal, magic, cubic and bistochastic, see \cite{bsp}. Now with these notions in hand, our third result is as follows:

\medskip
\noindent {\bf Theorem C.} {\em For $G=O^+,S^+,H^+,B^+$ let $C_+(G_n/G_k)$ be the universal $C^*$-algebra generated by the entries of a transposed $n\times(n-k)$ orthogonal, magic, cubic, stochastic isometry. Then $C_+(G_n/G_k)$ has the same abelianization and reduced version as $C_\times(G_n/G_k)$.}
\medskip

We refer to section 5 below for the precise statement. The proof uses the integration formula of Collins-\'Sniady \cite{csn}, as extended in \cite{bsp}, and a method from \cite{bgo}.

The problem of computing the kernel of $C_+(G_n/G_k)\to C_\times(G_n/G_k)$, which would be useful in connection with (1,2,3), seems to be a difficult linear algebra one, somehow in the spirit of those solved in \cite{cht}. There might be also some connections with \cite{cur}, \cite{so1}.

We also believe that the situation $C_+(G/H)\to C_\times(G/H)\subset C(G/H)$ can appear in more general contexts, and might help for the general understanding of noncommutative homogeneous spaces. We have several questions here, formulated in section 6 below.

The paper is organized as follows: 1-2 contain generalities about the spaces of type $G/(G\cap U_k^+)$, including the group dual case result, and in 3-4 we investigate the ``easy'' case. The final sections, 5-6, contain further results, and a few concluding remarks.

\subsection*{Acknowledgements.} This work was partly done during a visit of T.B. at the IMPAN, a visit of T.B. and P.S. at the ESI in Vienna, and a visit of A.S. at the Cergy University. T.B. was supported by the ANR grants ``Galoisint'' and ``Granma'', A.S. and P.S. were supported by the National Science Center (NCN) grant no.
2011/01/B/ST1/05011, and P.S. was partially supported by European Union grant PIRSES-GA-2008-230836.

\section{Quotient spaces}

We use the quantum group formalism developed by Woronowicz in \cite{wo1}, \cite{wo2}, \cite{wo3}. First, a ``Hopf $C^*$-algebra'' is a unital $C^*$-algebra $A$, together with a morphism of $C^*$-algebras $\Delta:A\to A\otimes A$, satisfying the coassociativity condition $(\Delta\otimes id)\Delta=(id\otimes\Delta)\Delta$ and the cobisimplifiability condition $\overline{span}\Delta(A)(A\otimes 1)=\overline{span}\Delta(A)(1\otimes A)=A\otimes A$. A morphism of Hopf $C^*$-algebras is a $C^*$-algebra morphism intertwining the comultiplications.

Given such a Hopf $C^*$-algebra $A$, we write $A=C(G)$, and we call $G$ ``compact quantum group''. There are two basic examples of compact quantum groups:
\begin{enumerate}
\item The compact groups $G$. Here the comultiplication is given by $\Delta f(g,h)=f(gh)$, by using the standard identification $C(G\times G)=C(G)\otimes C(G)$.

\item The discrete group duals $G=\widehat{\Gamma}$. Here $C(G)$ is by definition the group algebra $C^*(\Gamma)$, and the comultiplication is given by $\Delta(g)=g\otimes g$, for any $g\in\Gamma$.
\end{enumerate}

Note that in the above examples the square of the antipode is the identity. In fact, this condition will be satisfied by all quantum groups to be considered in this paper.

\begin{definition}
A left coaction of a compact quantum group $G$ on a unital $C^*$-algebra $A$ is a morphism of $C^*$-algebras $\alpha:A\to C(G)\otimes A$ satisfying the following conditions:
\begin{enumerate}
\item Coassociativity condition: $(id\otimes\alpha)\alpha=(\Delta\otimes id)\alpha$.

\item Non-degeneracy: $\alpha(A)(C(G)\otimes 1)$ is dense in $C(G)\otimes A$.
\end{enumerate}
\end{definition}

In the case where $G$ is a usual group, and the algebra $A=C(X)$ is commutative, such a coaction must be of the form $\alpha f(g,x)=f(g(x))$, for a certain action of $G$ on $X$.

In general, the basic example of coaction is the comultiplication map $\Delta$, corresponding to the ``action of $G$ on itself''. Observe that, in the particular case where $G$ is a classical group, this coaction comes indeed from the action $g(h)=gh$ of $G$ on itself.

More generally, assume that $H\subset G$ is a quantum subgroup, in the sense that we have a surjective morphism of Hopf $C^*$-algebras $\pi:C(G)\to C(H)$. Then the formula $\alpha=(\pi\otimes id)\Delta$ defines a coaction $\alpha:C(G)\to C(H)\otimes C(G)$. In the particular case where $G$ is a classical group, this coaction comes from the action $h(g)=hg$ of $H$ on $G$.

We recall that the fixed point algebra of a coaction is $A^\alpha=\{f\in A|\alpha f=1\otimes f\}$. Since in the classical case the condition $\alpha f(g,x)=(1\otimes f)(g,x)$ reads $f(g(x))=f(x)$, in this case we have an identification $C(X)^\alpha=C(X/G)$, where $X/G=\{Gx|x\in X\}$.

\begin{definition}
Let $H\subset G$ be a quantum subgroup of a compact quantum group, so that we have a surjective morphism of Hopf $C^*$-algebras $\pi:C(G)\to C(H)$. We let
$$C(G/H)=\{f\in C(G)|(\pi\otimes id)\Delta f=1\otimes f\}$$
be the fixed point algebra of the canonical coaction $\alpha:C(G)\to C(H)\otimes C(G)$.
\end{definition}

As a first remark, in the case where $G$ is a classical group, $C(G/H)$ is indeed the algebra of continuous functions on the quotient space $G/H$. Indeed, according to the above discussion, $C(G/H)$ consists of the functions $f\in C(G)$ which are invariant under the action $h(g)=hg$ of $H$ on $G$, i.e. which satisfy $f(g)=f(hg)$, for any $g\in G$ and $h\in H$. But this is the same as saying that $f$ has to be constant on each right coset $Hg$, so the algebra formed by these fuctions is simply $C(G/H)$, where $G/H=\{Hg|g\in G\}$.

\begin{proposition}
The noncommutative space $G/H$ is as follows:
\begin{enumerate}
\item For an inclusion of compact groups $H\subset G$ we have $G/H=\{Hg|g\in G\}$.

\item For a group dual inclusion $\widehat{\Lambda}\subset\widehat{\Gamma}$ we have $\widehat{\Gamma}/\widehat{\Lambda}=\widehat{\Theta}$, with $\Theta=\ker(\Gamma\to\Lambda)$.
\end{enumerate}
\end{proposition}

\begin{proof}
(1) This follows from the above discussion.

(2) Assume indeed that we have a surjective group morphism $\pi:\Gamma\to\Lambda$, and consider its canonical extension $\tilde{\pi}:C^*(\Gamma)\to C^*(\Lambda)$, which can be viewed as a morphism $\tilde{\pi}:C(\widehat{\Gamma})\to C(\widehat{\Lambda})$. Then for any element $f=\sum\lambda_g\cdot g$ of the algebra on the left, we have:
\begin{eqnarray*}
(\tilde{\pi}\otimes id)\Delta(f)
&=&\sum_{g\in\Gamma}\lambda_g\cdot\pi(g)\otimes g\\
1\otimes f
&=&\sum_{g\in\Gamma}\lambda_g\cdot 1\otimes g
\end{eqnarray*}

Thus we have $f\in C(\widehat{\Gamma}/\widehat{\Lambda})$ if and only if $\lambda_g\pi(g)=\lambda_g1$ for any $g\in\Gamma$, i.e. if and only if the support of $f$ is contained in the group $\Theta=\ker(\pi)$, and this gives the result.
\end{proof}

In order to understand the canonical action of $G$ on $G/H$, we need to introduce the notion of right coaction: this is by definition a morphism of $C^*$-algebras $\alpha:A\to A\otimes C(G)$ satisfying $(\alpha\otimes id)\alpha=(id\otimes\Delta)\alpha$, and such that $\alpha(A)(1\otimes C(G))$ is dense in $A\otimes C(G)$.

Assuming that we have such a coaction, we say that the underlying noncommutative space has a ``unique $G$-invariant probability measure'' if there exists a unique positive faithful unital state $\varphi:A\to\mathbb C$ satisfying the invariance condition $(\varphi\otimes id)\alpha=\varphi(.)1$.

\begin{proposition}
The noncommutative space $G/H$ has the following properties:
\begin{enumerate}
\item The right action of $G$ on itself restricts to an action on $G/H$.

\item There is a unique $G$-invariant probability measure on $G/H$.
\end{enumerate}
\end{proposition}

\begin{proof}
This is known from Podle\'s \cite{pod} and Boca \cite{boc}, here are just a few details:

(1) We first check that the algebra $C(G/H)\subset C(G)$ is invariant under $\Delta$, i.e. that we have $\Delta(C(G/H))\subset C(G/H)\otimes C(G)$. So, let $f\in C(G/H)$. We have:
\begin{eqnarray*}
(((\pi\otimes id)\Delta)\otimes id)\Delta f
&=&(\pi\otimes id\otimes id)(\Delta\otimes id)\Delta f\\
&=&(\pi\otimes id\otimes id)(id\otimes\Delta)\Delta f\\
&=&(id\otimes\Delta)(\pi\otimes id)\Delta f\\
&=&(id\otimes\Delta)(1\otimes f)\\
&=&1\otimes\Delta f
\end{eqnarray*}

Thus $\Delta f\in C(G/H)\otimes C(G)$. For the non-degeneracy axiom, see Podle\'s \cite{pod}.

(2) For the existence part, we can simply take $\varphi$ to be the restriction of the Haar state $h:C(G)\to\mathbb C$, constructed by Woronowicz in \cite{wo1}. For the uniqueness now, we have:
$$\varphi(x)=h(\varphi(x)1)=h((\varphi\otimes id)\alpha(x))=\varphi((id\otimes h)\alpha(x))=\varphi(h(x)1)=h(x)$$

Thus we have $\varphi=h$, and this finishes the proof.
\end{proof}

We investigate now the functoriality properties of the operation $(G,H)\to G/H$.

Two closed subgroups $H,H'\subset G$ are called isomorphic, and we write $H\simeq H'$, if there exists a Hopf $C^*$-algebra isomorphism $\theta:C(H)\to C(H')$ such that $\theta\pi=\pi'$.

\begin{theorem}
Let $H,H'\subset G$ be quantum subgroups of a compact quantum group.
\begin{enumerate}
\item We have $H\simeq H'$ if and only if $C(G/H)=C(G/H')$.

\item If $H'\subset H$ is proper, then $C(G/H)\subset C(G/H')$ is proper.
\end{enumerate}
\end{theorem}

\begin{proof}
For a compact quantum group $G$ we denote by $R(G)$ its algebra of representative functions, i.e. of coefficients of finite dimensional unitary representations of $G$.

We use the standard fact that the fixed point algebra of a coaction $\alpha:A\to C(G)\otimes A$ is given by $A^\alpha=(h\otimes id)\alpha(A)$, where $h:C(G)\to\mathbb C$ is the Haar state.

(1) If $H\simeq H'$ then we trivially have $C(G/H)=C(G/H')$. So, assume that $C(G/H)=C(G/H')$. If $G$ was coamenable, the proof would have followed directly from \cite{fst} and \cite{psk} (or \cite{so2}) as follows: the algebra $C(G/H)$ is an expected $C^*$-subalgebra of $C(G)$, with the conditional expectation onto $C(G/H)$ preserving the Haar state given by the formula $E_\omega=(\omega\otimes id)\Delta$, with $\omega=h_H\pi$. The algebra $C(H)$ (note that $H$ is automatically coamenable, see for example \cite{sal}) is then isomorphic, as the algebra of continuous functions on a quantum group, to the algebra $C(G)/N_\omega$, where $N_\omega=\{f\in C(G)|\omega(f^*f)=0\}$. Now it suffices to observe that as the Haar state on a coamenable quantum group is faithful, the conditional expectation onto $C(G/H)$ preserving $h_G$ is uniquely determined. Hence the equality $C(G/H)=C(G/H')$ implies that for the relevant conditional expectations we have $E_\omega=E_{\omega'}$, and applying the counit, $\omega=\omega'$. Hence $C(G)/N_{\omega}=C(G)/N_{\omega'}$ and the proof of the special case where $G$ is coamenable is finished.

In the general case now, assume that $C(G/H)=C(G/H')$, and assume first that we have $\omega\neq\omega'$. Then $E_\omega\neq E_{\omega'}$ and by density $E_\omega|_{R(G)}\neq E_{\omega'}|_{R(G)}$. Denote the images of the expectations on both sides of the last formula by $B,B'$, and let $b\in B-B'$. As $E_{\omega'}|_{R(G)}$ is a conditional expectation onto $B'$, we can assume that $b\neq 0$ and $E_{\omega'}(b) = 0$. As $E_{\omega'}$ preserves the Haar state $h_G$, we must have $h_G(b^*c')=0$ for all $c' \in B'$. But the assumption $C(G/H)=C(G/H')$ implies that $b$ can be approximated in norm by the elements in $B'$. Thus $h_G(b^*b)=0$, which contradicts the faithfulness of $h_G$ on $R(G)$.

Thus we have $\omega=\omega'$. Now by using for instance the arguments in Theorem 3.7 of \cite{psk} we can show that $R(H)$ is isomorphic to the quotient of $R(G)$ by $R(G) \cap N_{\omega}$. From that we see that $H\simeq H'$, and we are done.

(2) This follows from (1).
\end{proof}

For further information on noncommutative homogeneous spaces, see Podle\'s \cite{pod}.

\section{The matrix case}

As explained in the introduction, given a compact group $G\subset U_n$ and a number $k\leq n$, we can consider the compact group $H=G\cap U_k$, where the embedding $U_k\subset U_n$ is given by $g\to diag(g,1_{n-k})$, and then form the homogeneous space $X=G/H$.

\begin{proposition}
Let $G\subset U_n$ be a closed subgroup, and let $H=G\cap U_k$. Then $C(G/H)$, viewed as subalgebra of $C(G)$, is generated by the last $n-k$ rows of coordinates on $G$.
\end{proposition}

\begin{proof}
Let $u_{ij}\in C(G)$ be the standard coordinates on $G$, given by $u_{ij}(g)=g_{ij}$, and let $A$ be the algebra generated by the functions $\{u_{ij}|i>k,j>0\}$. Since each $u_{ij}$ with $i>k$ is constant on each coset $Hg\in G/H$, we have an inclusion $A\subset C(G/H)$.

In order to prove that this inclusion in a isomorphism, we use the Stone-Weierstrass theorem: it is enough to show that the functions $\{u_{ij}|i>k,j>0\}$ separate the cosets $\{Hg|g\in G\}$. But this is the same as saying that $Hg\neq Hh$ implies that $g_{ij}\neq h_{ij}$ for some $i>k,j>0$, or, equivalently, that $g_{ij}=h_{ij}$ for any $i>k,j>0$ implies that we have $Hg=Hh$. Now since $Hg=Hh$ is equivalent to $gh^{-1}\in H$, the result follows from the usual matrix formula of $gh^{-1}$, and from the fact that $g,h$ are unitary.
\end{proof}

In order to deal with the quantum case, we use a construction of Wang \cite{wa1}. Let us call ``biunitary'' any unitary matrix $u=(u_{ij})$, whose transpose $u^t=(u_{ji})$ is also unitary.

\begin{definition}
The quantum group $U_n^+$ is the abstract dual of $C(U_n^+)$, the universal $C^*$-algebra generated by the entries of a $n\times n$ biunitary matrix.
\end{definition}

Observe that $C(U_n^+)$ is a Hopf $C^*$-algebra in the sense of Woronowicz \cite{wo1}, \cite{wo2}, with comultiplication $\Delta(u_{ij})=\sum_ku_{ik}\otimes u_{kj}$, counit $\varepsilon(u_{ij})=\delta_{ij}$ and antipode $S(u_{ij})=u_{ji}^*$. Note that the square of the antipode is the identity, $S^2=id$. See Wang \cite{wa1}.

As in the classical case, let $k\leq n$, and consider the embedding $U_k^+\subset U_n^+$ given by $g\to diag(g,1_{n-k})$. That is, at the level of algebras, we use the surjective morphism $C(U_n^+)\to C(U_k^+)$ obtained by mapping the fundamental corepresentation of $U_n^+$ to the matrix $diag(v,1_{n-k})$, where $v$ is the fundamental corepresentation of $U_k^+$.

\begin{definition}
Associated to any quantum subgroup $G\subset U_n^+$ and any $k\leq n$ are:
\begin{enumerate}
\item The compact quantum group $H=G\cap U_k^+$.

\item The algebra $C(G/H)\subset C(G)$ constructed in section 1.

\item The algebra $C_\times(G/H)\subset C(G/H)$ generated by $\{u_{ij}|i>k,j>0\}$.
\end{enumerate}
\end{definition}

This definition uses a number of simple facts, to be explained now. First, we use in (1) above the intersection operation for quantum subgroups, defined as follows: if $H,H'\subset G$ are quantum subgroups, we let $K=H\cap H'$ be the quantum subgroup of $G$ defined by the fact that $\ker(C(G)\to C(K))$ is the sum of the ideals $\ker(C(G)\to C(H))$ and $\ker(C(G)\to C(H'))$. Observe that we have the following commuting diagram:
$$\begin{matrix}
C(G)&\to&C(H)\\
\\
\downarrow&&\downarrow\\
\\
C(H')&\to&C(K)
\end{matrix}$$

In diagrammatic terms, the universal property of $K=H\cap H'$ is that for any square diagram like the one above, but with an arbitrary $C^*$-algebra $C(K')$ in place of $C(K)$ there exists a unique map $\rho:C(K)\to C(K')$ making the obvious diagram commutative. In other words $C(K)$ is a push-out in the sense of \cite{ped}. It is easy to see that $K$ is naturally endowed with a compact quantum group
structure, making is a subgroup of $H,H'$.

Regarding now (3), let $u,v$ be the fundamental corepresentations of $G,H$, so that the arrow $\pi:C(G)\to C(H)$ is given by $u\to diag(v,1_{n-k})$. Then $(\pi\otimes id)\Delta$ is given by:
$$(\pi\otimes id)\Delta(u_{ij})
=\sum_s\pi(u_{is})\otimes u_{sj}
=\begin{cases}
\sum_{s\leq k}v_{is}\otimes u_{sj}&i\leq k\\
1\otimes u_{ij}&i>k
\end{cases}$$

In particular we see that the equality $(\pi\otimes id)\Delta f=1\otimes f$ defining $C(G/H)$ holds on all the coefficients $f=u_{ij}$ with $i>k$, and this justifies the inclusion appearing in (3).

As a first remark, it can be shown that $C_\times(G/H)$ is an embeddable quantum homogeneous space algebra in the sense of Podle\'s \cite{pod}.

Let us first try to understand what happens in the group dual case. We will do our study here in two steps: first in the ``diagonal'' case, and then in the general case.

We recall that given a discrete group $\Gamma=<g_1,\ldots,g_n>$, the matrix $D=diag(g_i)$ is biunitary, and produces a surjective morphism $C(U_n^+)\to C^*(\Gamma)$. This morphism can be viewed as corresponding to a quantum embedding $\widehat{\Gamma}\subset U_n^+$, that we call ``diagonal''.

The normal closure of a subgroup $\Lambda\subset\Gamma$ is the biggest subgroup $\Lambda'\subset\Gamma$ containing $\Lambda$ as a normal subgroup. Note that $\Lambda'$ can be different from the normalizer $N(\Lambda)$.

\begin{proposition}
Assume that $G=\widehat{\Gamma}$, with $\Gamma=<g_1,\ldots,g_n>$, diagonally embedded, and let $H=G\cap U_k^+$.
\begin{enumerate}
\item $H=\widehat{\Theta}$, where $\Theta=\Gamma/<g_{k+1}=1,\ldots,g_n=1>$.

\item $C_\times(G/H)=C^*(\Lambda)$, where $\Lambda=<g_{k+1},\ldots,g_n>$.

\item $C(G/H)=C^*(\Lambda')$, where ``prime'' is the normal closure.

\item $C_\times(G/H)=C(G/H)$ if and only if $\Lambda\triangleleft\Gamma$.
\end{enumerate}
\end{proposition}

\begin{proof}
We use the standard fact that for any group $\Gamma=<a_i,b_j>$, the kernel of the quotient map $\Gamma\to\Gamma/<a_i=1>$ is the normal closure of the subgroup $<a_i>\subset\Gamma$.

(1) Since the map $C(U_n^+)\to C(U_k^+)$ is given on diagonal coordinates by $u_{ii}\to v_{ii}$ for $i\leq k$ and $u_{ii}\to 1$  for $i>k$, the result follows from definitions.

(2) Once again, this assertion follows from definitions.

(3) From Proposition 1.3 and from (1) we get $G/H=\widehat{\Lambda'}$, where $\Lambda'=\ker(\Gamma\to\Theta)$. By the above observation, this kernel is exactly the normal closure of $\Lambda$.

(4) This follows from (2) and (3).
\end{proof}

Let us try now to understand the general group dual case. We recall that the subgroups $\widehat{\Gamma}\subset U_n^+$ appear as follows: $\Gamma=<g_1,\ldots,g_n>$ is a discrete group, $J\in U_n$ is a unitary, and the  morphism $C(U_n^+)\to C^*(\Gamma)$ is given by $u\to JDJ^*$, where $D=diag(g_i)$.

This follows indeed from two results of Woronowicz, namely: (1) the finite-dimensional unitary representations of any compact quantum group, and in particular of $\widehat{\Gamma}$, are completely reducible, and (2) the irreducible representations of $\widehat{\Gamma}$ are all 1-dimensional, and correspond to the group elements $g\in\Gamma$. Indeed, we obtain that the $n$-dimensional unitary representations of $\widehat{\Gamma}$ are precisely those of the form $JDJ^*$, with $D=diag(g_i)$ a direct sum of irreducibles, and with $J\in U_n$ unitary, as stated above. See Theorem 1.7 in \cite{wo1}.

With this result in hand, Proposition 2.4 generalizes as follows:

\begin{theorem}
Assume that $G=\widehat{\Gamma}$, with $\Gamma=<g_1,\ldots,g_n>$, embedded via $u\to JDJ^*$, and let $H=G\cap U_k^+$.
\begin{enumerate}
\item $H=\widehat{\Theta}$, where $\Theta=\Gamma/<g_r=1|\exists\,i>k,J_{ir}\neq 0>$, embedded $u_{ij}\to (JDJ^*)_{ij}$.

\item $C_\times(G/H)=C^*(\Lambda)$, where $\Lambda=<g_r|\exists\,i>k,J_{ir}\neq 0>$.

\item $C(G/H)=C^*(\Lambda')$, where ``prime'' is the normal closure.

\item $C_\times(G/H)=C(G/H)$ if and only if $\Lambda\triangleleft\Gamma$.
\end{enumerate}
\end{theorem}

\begin{proof}
We will basically follow the proof of Proposition 2.4 above.

(1) Let $\Lambda=<g_1,\ldots,g_n>$, let $J\in U_n$, and consider the embedding $\widehat{\Lambda}\subset U_n^+$ corresponding to the morphism $C(U_n^+)\to C^*(\Lambda)$ given by $u\to JDJ^*$, where $D=diag(g_i)$.

Let $G=\widehat{\Lambda}\cap U_k^+$. Since we have $G\subset\widehat{\Lambda}$, the algebra $C(G)$ is cocomuttaive, so we have $G=\widehat{\Theta}$ for a certain discrete group $\Theta$. Moreover, the inclusion $\widehat{\Theta}\subset\widehat{\Lambda}$ must come from a group morphism $\varphi:\Lambda\to\Theta$. Also, since $\widehat{\Theta}\subset U_k^+$, we have a morphism $C(U_k^+)\to C^*(\Theta)$ given by $v\to V$, where $V$ is a certain $k\times k$ biunitary over $C^*(\Theta)$.

With these observations in hand, let us look now at the intersection operation. We must have a group morphism $\varphi:\Lambda\to\Theta$ such that the following diagram commutes:
$$\begin{matrix}
C(U_n^+)&\to&C(U_k^+)\\
\\
\downarrow&&\downarrow\\
\\
C^*(\Lambda)&\to&C^*(\Theta)
\end{matrix}$$

Thus we must have $(id\otimes\varphi)(JDJ^*)=diag(V,1_{n-k})$, and with $f_i=\varphi(g_i)$, we get:
$$\sum_rJ_{ir}\bar{J}_{jr}f_r
=\begin{cases}
V_{ij}&\mbox{if }i,j\leq k\\
\delta_{ij}&\mbox{otherwise}
\end{cases}$$

Now since $J$ is unitary, the second part of the above condition is equivalent to ``$f_r=1$ whenever there exists $i>k$ such that $J_{ir}\neq 0$''. Indeed, this condition is easily seen to be equivalent to the ``$=1$'' conditions, and implies the ``$=0$'' conditions. We claim that:
$$\Theta=\Lambda/<g_r=1|\exists\,i>k,J_{ir}\neq 0>$$

Indeed, the above discussion shows that $\Theta$ must be a quotient of the group on the right, say $\Theta_0$. On the other hand, since in $C^*(\Theta_0)$ we have $J_{ir}g_r=J_{ir}1$ for any $i>k$, we obtain that $(JDJ^*)_{ij}=\delta_{ij}$ unless $i,j\leq k$, so we have $JDJ^*=diag(V,1_{n-k})$, for a certain matrix $V$. But $V$ must be a biunitary, so we have a morphism $C(U_k^+)\to C^*(\Theta_0)$ mapping $v\to V$, which completes the push-out diagram, and proves our claim.

(2) Let $A_{ij}=\sum_rJ_{ir}\bar{J}_{jr}g_r$ with $i>k,j>0$ be the standard generators of $C_\times(G/H)$. Since $\sum_jA_{ij}J_{jm}=J_{im}g_m$ we conclude that $C_\times(G/H)$ contains any $g_r$ such that there exists $i>k$ with $J_{ir}\neq 0$, i.e. contains any $g_r\in\Lambda$. Conversely, if $g_r\in\Gamma-\Lambda$ then $J_{ir}g_r=0$ for any $i>k$, so $g_r$ doesn't appear in the formula of any of the generators $A_{ij}$.

(3,4) The proof here is similar to the proof of Proposition 2.4 (3,4).
\end{proof}

We should mention that Theorem 2.5 has some further extensions, when plugged into the general machinery developed by Podle\'s in \cite{pod}. One can prove for instance that, given a discrete group $\Gamma$, there is a natural bijection between the set of embeddable quantum homogeneous space algebras $A\subset C^*(\Gamma)$ and the lattice of subgroups of $\Gamma$. Under this bijection the quotient quantum homogeneous spaces correspond to normal subgroups.

\section{The easy case}

In the reminder of this paper we study the algebras $C_\times(G/H)\subset C(G/H)$ constructed in Definition 2.3, in the case where $G$ is an ``easy quantum group'', in the sense of \cite{bsp}. The motivations here come from several questions, partly mentioned in the introduction, and which will be further explained in this section, and in the following two ones.

Let us first recall the definition of the easy groups \cite{bsp}. We recall that the space of fixed points of a representation $u$ is given by $Fix(u)=\{\xi|u(\xi\otimes 1)=\xi\otimes 1\}$.

Let $\{e_1,\ldots,e_n\}$ be the standard basis of $\mathbb C^n$. We have the following construction:

\begin{definition}
The partitions $\pi\in P(s)$ produce vectors of $(\mathbb C^n)^{\otimes s}$ via the formula
$$\xi_\pi=\sum_{i_1\ldots i_s}\delta_\pi(i)e_{i_1}\otimes\ldots\otimes e_{i_s}$$
where $\delta_\pi\in\{0,1\}$ is defined as follows: we put the indices $i_1,\ldots,i_s$ on the points of $\pi$, and we set $\delta_\pi=1$ if any block of $\pi$ contains only equal indices of $i$, and $\delta_\pi=0$ if not.
\end{definition}

Consider now the embedding $S_n\subset O_n$ given by the usual permutation matrices, and let $S_n\subset G_n\subset O_n$ be an intermediate compact group. If $u,v,w$ denote respectively the fundamental representations of $S_n,G_n,O_n$, by functoriality we have embeddings of fixed point spaces $Fix(w^{\otimes s})\subset Fix(v^{\otimes s})\subset Fix(u^{\otimes s})$, for any $s\in\mathbb N$. Now by some well-known results, basically going back to Brauer's work in \cite{bra}:
\begin{enumerate}
\item The space $Fix(u^{\otimes s})$ is spanned by $P(s)$, the set of partitions of $s$ points.

\item The space $Fix(w^{\otimes s})$ is spanned by $P_2(s)$, the set of pairings of $s$ points.
\end{enumerate}

As a first consequence, we can conclude that the abstract spaces $Fix(v^{\otimes s})$ are not that abstract: they consist in fact of linear combinations of partitions in $P(s)$.

The ``easiest'' case, combinatorially speaking, is when each $Fix(v^{\otimes s})$ is spanned by certain partitions in $P(s)$. This leads to the following definition, first stated in \cite{bsp}:

\begin{definition}
A compact group $S_n\subset G_n\subset O_n$ is called ``easy'' if there exist sets of diagrams $P_2(s)\subset D(s)\subset P(s)$ such that $Fix(v^{\otimes s})=span(D(s))$, for any $s\in\mathbb N$.
\end{definition}

As a first example, the groups $O_n,S_n$ are easy, with $D=P_2,P$ respectively. Some other examples are the hyperoctahedral and bistochastic groups $H_n,B_n$, where $D=P_{even},P_{12}$ (partitions with all blocks having even size, respectively singletons and pairings).

It was proved in \cite{bsp} that, besides these 4 main examples, there are only 2 more easy groups, namely $S_n'=\mathbb Z_2\times S_n$ and $B_n'=\mathbb Z_2\times B_n$. Here the sets of partitions are $D=P_{even}',P_{12}'$, where the ``prime'' operation consists in setting $D(s)=\emptyset$ for $s$ odd.

More generally, in the quantum case now, we have the following definition.

\begin{definition}
A quantum group $S_n\subset G_n\subset O_n^+$ is called ``easy'' if there exist sets of diagrams $NC_2(s)\subset D(s)\subset P(s)$ such that $Fix(v^{\otimes s})=span(D(s))$, for any $s\in\mathbb N$.
\end{definition}

Here $O_n^+$ is the free orthogonal quantum group, corresponding to the universal Hopf algebra $A_o(n)$ constructed by Wang in \cite{wa1}. As for $NC_2(s)$, this is the set of noncrossing pairings of $s$ points, known to span the fixed point spaces for $O_n^+$. See \cite{bsp}.

The basic examples are the ``free analogues'' of the above-mentioned 6 easy groups. These are quantum groups $O_n^+,S_n^+,H_n^+,B_n^+,S_n'^+,B_n'^+$, constructed in Wang's papers \cite{wa1}, \cite{wa2} and then in \cite{ahn}, \cite{bsp}, and whose sets of partitions are obtained from the corresponding classical sets, by intersecting with $NC(s)$. It was proved in \cite{bsp} that in the ``free case'', i.e. when we have $S_n^+\subset G_n$, these 6 examples are the only ones.

It is important to note that, in these classification results in the classical and free cases, all the examples come in series, i.e. there is no ``exceptional'' example. See \cite{bsp}.

Some other examples, which are neither classical, nor free, were constructed in \cite{ez1}.

The easy groups have led to some applications to probability and free probability, and one key property here, introduced and heavily used in \cite{ez2}, is as follows.

\begin{definition}
An easy quantum group, with category of partitions $D=(D(s))$, is called ``multiplicative'' if $D$ is stable by the operation consisting in removing blocks.
\end{definition}

As a first example, the groups $O_n,S_n,H_n,B_n$ are clearly multiplicative, and so are their free versions $O_n^+,S_n^+,H_n^+,B_n^+$. On the negative side, the groups $S_n',B_n'$ are clearly not multiplicative, nor are their free versions $S_n'^+,B_n'^+$. As for the ``intermediate'' quantum groups constructed in \cite{ez1}, these are in general not multiplicative either.

As already mentioned, the multiplicativity assumption is needed in order to have some control over the combinatorics of $D$, as to obtain probabilistic results about $G$. Without getting here into details, let us just mention the following result: ``in the context of a liberation operation $G_n\to G_n^+$, the asymptotic spectral distributions of the main characters are in Bercovici-Pata bijection \cite{bpa} precisely in the multiplicative case''. See \cite{bsp}.

For the purposes of the present paper, the multiplicativity condition will play as well a central role. We have the following result, clarifying the algebraic meaning of this condition, and making the link with the abstract considerations in section 2 above.

\begin{theorem}
For an easy quantum group $G_n\subset U_n^+$, the following are equivalent:
\begin{enumerate}
\item $G_n$ is multiplicative in the above sense.

\item We have $G_n\cap U_k^+=G_k$, for any $k\leq n$.
\end{enumerate}
\end{theorem}

\begin{proof}
We will prove that $G_n\cap U_k^+=G_k'$, where $G'=(G_n')$ is the easy quantum group associated to the category $D'$ generated by all subpartitions of the partitions in $D$.

As explained in \cite{bsp}, the correspondence between categories of partitions and easy quantum groups comes from Woronowicz's Tannakian duality in \cite{wo2}. More precisely, the quantum group $G_n\subset O_n^+$ associated to a category of partitions $D=(D(s))$ is obtained by imposing to the fundamental representation of $O_n^+$ the fact that its $s$-th tensor power must fix $\xi_\pi$, for any $s\in\mathbb N$ and $\pi\in D(s)$. So, we have the following presentation result:
$$C(G_n)=C(O_n^+)/<\xi_\pi\in Fix(u^{\otimes s}),\forall s,\,\forall\pi\in D(s)>$$

Now since $\xi_\pi\in Fix(u^{\otimes s})$ means $u^{\otimes s}(\xi_\pi\otimes 1)=\xi_\pi\otimes 1$, this condition is equivalent to the following collection of equalities, one for each multi-index $i\in\{1,\ldots,n\}^s$:
$$\sum_{j_1\ldots j_s}\delta_\pi(j)u_{i_1j_1}\ldots u_{i_sj_s}=\delta_\pi(i)1$$

Summarizing, we have the following presentation result:
$$C(G_n)=C(O_n^+)/<\sum_{j_1\ldots j_s}\delta_\pi(j)u_{i_1j_1}\ldots u_{i_sj_s}=\delta_\pi(i)1,\forall s,\,\forall\pi\in D(s),\,\forall i>$$

Our first claim is as follows: let $k\leq n$, assume that we have a compact quantum group $K\subset O_k^+$, with fundamental representation denoted $u$, and consider the $n\times n$ matrix $\tilde{u}=diag(u,1_{n-k})$. Then for any $s\in\mathbb N$ and any $\pi\in P(s)$, we have:
$$\xi_\pi\in Fix(\tilde{u}^{\otimes s})\iff\xi_{\pi'}\in Fix(u^{\otimes s'}),\,\forall\pi'\subset\pi$$

Here $\pi'\subset\pi$ means that $\pi'\in P(s')$ is obtained from $\pi\in P(s)$ by removing some of its blocks. The proof of this claim is standard. Indeed, when making the replacement $u\to\tilde{u}$ and trying to check the condition $\xi_\pi\in Fix(\tilde{u}^{\otimes s})$, we have two cases:

-- $\delta_\pi(i)=1$. Here the $>k$ entries of $i$ must be joined by certain blocks of $\pi$, and we can consider the partition $\pi'\in D(s')$ obtained by removing these blocks. The point now is that the collection of $\delta_\pi(i)=1$ equalities to be checked coincides with the collection of $\delta_\pi(i)=1$ equalities expressing the fact that we have $\xi_\pi\in Fix(u^{\otimes s'})$, for any $\pi'\subset\pi$.

-- $\delta_\pi(i)=0$. In this case the situation is quite similar: the collection of $\delta_\pi(i)=0$ equalities to be checked coincides, modulo some $0=0$ identities, with the collection of $\delta_\pi(i)=0$ equalities expressing the fact that we have $\xi_\pi\in Fix(u^{\otimes s'})$, for any $\pi'\subset\pi$.

Our second claim is as follows: given a quantum group $K\subset O_n^+$, with fundamental representation denoted $v$, the algebra of functions on $H=K\cap O_k^+$ is given by:
$$C(H)=C(O_k^+)/<\xi\in Fix(\tilde{u}^{\otimes s}),\,\forall\xi\in Fix(v^{\otimes s})>$$

This follows indeed from Woronowicz's results in \cite{wo2}, because the algebra on the right comes from the Tannakian formulation of the intersection operation in Definition 2.3.

Now with the above two claims in hand,  we can conclude that we have $G_n\cap U_k^+=G_k'$, where $G'=(G_n')$ is the easy quantum group associated to the category $D'$ generated by all the subpartitions of the partitions in $D$. In particular we see that the condition $G_n\cap U_k^+=G_k^+$ for any $k\leq n$ is equivalent to $D=D'$, and this gives the result.
\end{proof}

\section{Properness results}

In this section we study the inclusions $C_\times(G_n/G_k)\subset C(G_n/G_k)$, where $G=(G_n)$ is a multiplicative easy quantum group. We recall that the basic examples are the classical groups $S,O,H,B$, and their free analogues $S^+,O^+,H^+,B^+$. In addition, it is known that in the free case the list of such quantum groups is precisely $S^+,O^+,H^+,B^+$. See \cite{bsp}.

Since in the classical case a complete answer is provided by Proposition 2.1 above, and in the ``intermediate'' case (i.e. not classical, nor free) we have no multiplicative examples, we will actually restrict attention to the free case. That is, we will assume that $G=(G_n)$ is one of the quantum groups $S^+,O^+,H^+,B^+$.

We will prove that the inclusions $C_\times(G_n/G_k)\subset C(G_n/G_k)$ are in general proper.

Let us first recall the defining relations between the coordinates of $G$.

\begin{proposition}
The defining relations for $C(G)$ in terms of the standard generators $u_{ij}$ are as follows:
\begin{enumerate}
\item $G=O_n^+$: $u$ is orthogonal, i.e. $u_{ij}$ are self-adjoint, and $u^t=u^{-1}$.

\item $G=S_n^+$: $u$ is magic, i.e. orthogonal, with $u_{ij}$ being projections.

\item $G=H_n^+$: $u$ is cubic, i.e. orthogonal, with $xy=0$ on rows and columns.

\item $G=B_n^+$: $u$ is bistochastic, i.e. orthogonal, with sum $1$ on rows and columns.
\end{enumerate}
\end{proposition}

We refer to \cite{bsk}, \cite{bsp} for a full discussion of these relations. 

Note that the name ``cubic'' comes at the same time from the fact that $H_n^+$ is the quantum symmetry group of the cube, cf. \cite{bbc}, and from the fact that the entries of the fundamental unitary representation of $H_n^+$ satisfy the condition $u_{ij}^3=u_{ij}$. 

In what follows we will use a number of simple facts regarding these relations. First, we have ``magic $=$ cubic + bistochastic'', which follows from definitions, by using some basic $C^*$-algebra tricks. This shows that we have the following inclusions:
$$\begin{matrix}
H_n^+&\subset&O_n^+\\
\\
\cup&&\cup\\
\\
S_n^+&\subset&B_n^+
\end{matrix}$$

The second observation, which is a bit more conceptual, is the fact we have the quantum group equalities $O_n^+=<H_n^+,B_n^+>$ and $S_n^+=H_n^+\cap B_n^+$. See \cite{bsk}.

We refer to section 5 below, and more specifically to Definition 5.1 and Proposition 5.2, for a partial generalization of these notions, in the ``rectangular'' framework.

Let us go back now to the inclusions $C_\times(G_n/G_k)\subset C(G_n/G_k)$. We first work out a few simple cases, where these inclusions are isomorphisms:

\begin{proposition}
The inclusion $C_\times(G_n/G_k)\subset C(G_n/G_k)$ is an isomorphism at $n=1$, at $k=0$, at $k=n$, as well as in the following special cases:
\begin{enumerate}
\item $G=B^+$: at $k=1$.

\item $G=S^+$: at $k=1$, and at $k=2,n=3$.
\end{enumerate}
\end{proposition}

\begin{proof}
First, the results at $n=1$, at $k=0$, and at $k=n$ are clear from definitions. Regarding now the special cases:

(1) Since the coordinates of $B_n^+$ sum up to 1 on each column, we have the formula $u_{1j}=1-\sum_{i>1}u_{ij}$, and so the inclusion $C_\times(B_n^+/B_1^+)\subset C(B_n^+)$ is an isomorphism. Thus the inclusion $C_\times(B_n^+/B_1^+)\subset C(B_n^+/B_1^+)$ must be as well an isomorphism.

(2) By using the same argument we obtain that the inclusion $C_\times(S_n^+/S_1^+)\subset C(S_n^+/S_1^+)$ is as well an isomorphism. In the remaining case $k=2,n=3$, or more generally at any $k\in\mathbb N$ and $n<4$, it is known from Wang \cite{wa2} that we have $S_n=S_n^+$, so by Proposition 2.1 the inclusion in the statement is $C(S_n/S_k)\subset C(S_n/S_k)$, and we are done again.
\end{proof}

In the opposite direction now, we have the following result:

\begin{theorem}
Let $G=(G_n)$ be a multiplicative free quantum group. Then the inclusion $C_\times(G_n/G_k)\subset C(G_n/G_k)$ is proper, for any $n\geq 4$, $2\leq k\leq n-1$.
\end{theorem}

\begin{proof}
First, according to the general results explained in section 3 above, the multiplicative free quantum groups are exactly the quantum groups $G=S^+,O^+,H^+,B^+$.

We denote by $u_{ij},v_{ij}$ the standard coordinates on $G_n,G_k$, so that the canonical surjective map $\pi:C(G_n)\to C(G_k)$ is given by:
$$\pi(u_{ij})
=\begin{cases}
v_{ij}&\textup{if }i,j\leq k\\
1&\textup{if }i=j>k\\
0&\textup{otherwise}
\end{cases}$$

The standard coaction $\alpha:C(G_n)\to C(G_k)\otimes C(G_n)$ is then given by:
$$\alpha(u_{ij})
=(\pi\otimes id)\Delta(u_{ij})
=\begin{cases}
\sum_{s\leq k}v_{is}\otimes u_{sj}&\textup{ if }i\leq k\\
1\otimes u_{ij}&\textup{ if }i>k
\end{cases}$$

Consider first the case $2\leq k \leq n-2$.

Fix a nontrivial projection $p\in C(\mathbb Z_2)$, and consider the following matrix:
$$\tilde{p}=\begin{pmatrix}p&p^\perp\\ p^\perp&p\end{pmatrix}$$

Consider also the algebra $A=C(G_k)*C(\mathbb Z_2)$, and let $\nu:C(G_n)\to A$ be the surjection induced by the matrix $diag(v,\tilde{p},1_{n-k-2})$, which satisfies the relations in Proposition 4.1.

Finally, consider the coaction $\beta:A\to C(G_k)\otimes A$ given by $\beta(p)=1\otimes p$ and:
$$\beta(v_{ij})=\sum_{s=1}^kv_{is}\otimes v_{sj}$$

We have $\beta\nu=(id\otimes\nu)\alpha$, and our first claim is that we have:
$$Fix(\beta)=\nu(Fix(\alpha))$$

Indeed, ``$\supset$'' is clear from $\beta\nu=(id\otimes\nu)\alpha$, and ``$\subset$'' follows as well from $\beta\nu=(id\otimes\nu)\alpha$, by using conditional expectations, because for $x\in Fix(\beta)$ we have:
$$x\in (h\otimes id)\beta(A)=(h\otimes id)\beta\nu(C(G_n))=(h\otimes\nu)\alpha(C(G_n))=\nu(Fix(\alpha))$$

Since $\nu(C_\times(G_n/G_k))=C(\mathbb Z_2)$, as subalgebras of $A$, it suffices to find an element in $Fix(\beta)$ which is not in $C(\mathbb Z_2)$. Let:
$$x=(h\otimes id)\beta(v_{11}pv_{11})=\frac{1}{k}\sum_{s=1}^kv_{s1}pv_{s1}$$

The last identity follows from the fact that for each $G_k$ we have: 
\[ h(v_{is}v_{js}) = \frac{1}{k} \delta_{ij}\]

As $x\in Fix(\beta)$, it remains to show that $x\notin\nu(C_\times(G_n/G_k))$. Consider the morphism $\rho=\eta*id:A \to C(\mathbb Z_2)*C(\mathbb Z_2)$, where $\eta:C(G_k)\to C(\mathbb Z_2)$ is induced by $diag(\tilde{q},1_{k-2})$, with $\tilde{q}$ defined analogously to $\tilde{p}$, which satisfies the relations in Proposition 4.1.

If $x\in\nu(C_\times(G_n/G_k))$, the element $x$ would have to commute with $p$. Similarly $\rho(x)$ would have to commute with $p'=\rho(p)$. But $\rho(x)=qp'q+q^\perp p'q^\perp$, where $q$ denotes the projection generating the first copy of $C(\mathbb Z_2)$ in $C(\mathbb Z_2)*C(\mathbb Z_2)$, and it is easy to see that $qp'q+q^\perp p'q^\perp$ does not commute with $p'$, for instance by working with a concrete model of $C(\mathbb Z_2)*C(\mathbb Z_2)$ given by $C^*(\mathbb Z_2*\mathbb Z_2)$. Thus $x\notin\nu(C_\times(G_n/G_k))$, and we are done.

Let now $k=n-1$ and put:
$$y=(h\otimes id)\alpha(u_{k k}u_{nn}u_{kk})
=\frac{1}{k}\sum_{s=1}^{k}u_{sk}u_{nn}u_{sk}$$

Then $y\in C(G_n^+/G_k^+)$, and we need to show that $y$ is not in $C_\times(G_n^+/G_k^+)$. By a passing to a subgroup argument we see it suffices to do it for $G=S^+$. Assume then that we are in this case. As we know that $C_\times(S_n^+/S_k^+)$ is commutative, it suffices to show that $y$ does not commute with $u_{nn}$. So, consider the surjection $\rho':C(S_n^+)\to C(\mathbb Z_2)*C(\mathbb Z_2)$ given by the following magic unitary matrix:
$$M=\begin{pmatrix}
1_{n-4} &0&0&0&0\\
0&p&0&p^\perp\\
0&0&q&0&q^\perp\\
0&p^\perp&0&p&0\\
0&0&q^\perp&0&q\\
\end{pmatrix}$$

Here $p$ and $q$ the free projections generating $C(\mathbb Z_2)*C(\mathbb Z_2)$. Then $\rho'(u_{nn})=q$, $\rho'(y)=p^\perp qp^\perp+ pqp$, and we can finish as in the previous case.
\end{proof}

Observe that in the above proof, the idea was to use a suitable group dual subgroup $\widehat{\Gamma}\subset S_n^+$. Note that these group dual subgroups were fully classified by Bichon in \cite{bic}.

One question that we have is whether one can deduce Theorem 4.3 directly from Theorem 2.5. We would need here a positive answer to the following problem:

\begin{problem}
Let $G\subset U_n^+$, let $k\leq n$, set $H=G\cap U_k^+$, and assume that $C_\times(G/H)=C(G/H)$. Then for any $G'\subset G$ we have $C_\times(G'/H')=C(G'/H')$, where $H'=G'\cap H$.
\end{problem}

Observe that this is indeed true in the classical case, and also in the group dual case, because we have $\Lambda\triangleleft\Gamma\implies\Lambda'\triangleleft\Gamma'$, for any surjective group morphism $g\to g'$. In the general case the problem is to prove that the map $C(G/H)\to C(G'/H')$ is surjective.

\section{Universal algebras}

Let $G=O^+,S^+,H^+,B^+$ be one of the 4 multiplicative free quantum groups. In this section we further investigate these row algebras $C_\times(G_n/G_k)$, by regarding them as quotients of certain universal algebras. For $G=O^+$ the motivation comes from the ``free spheres'' introduced in \cite{bgo}, and also, indirectly, from the $G=O$ computations in \cite{bsc}. For $G=S^+$ the row algebras $C_\times(G_n/G_k)$ appear to be quite subtle combinatorial objects, and the motivation comes from the various results in \cite{bbc}, \cite{cht}, \cite{cur}, \cite{so1}.

The axiomatization of the algebras $C_\times(G_n/G_k)$ is a quite tricky task, because these algebras have a rectangular matrix of generators, which is a transposed isometry, but not much is known about the remaining conditions to be satisfied by the generators.

We have here the following definition, inspired by Proposition 4.1 above:

\begin{definition}
Associated to $k\leq n$ is the universal $C^*$-algebra $C_+(G_n/G_k)$ generated by the entries of a rectangular matrix $p=(p_{ij})_{i>k,j>0}$, subject to the following conditions:
\begin{enumerate}
\item $G=O_n^+$: $p$ is a transposed ``orthogonal isometry'', in the sense that its entries $p_{ij}$ are self-adjoint, and $pp^t=1$.

\item $G=S_n^+$: $p$ is a transposed ``magic isometry'', in the sense that $p^t$ is an orthogonal isometry, and $p_{ij}$ are projections, orthogonal on columns.

\item $G=H_n^+$: $p$ is a transposed ``cubic isometry'', in the sense that $p^t$ is an orthogonal isometry, with $xy=0$ for any $x\neq y$ on the same row of $p$

\item $G=B_n^+$: $p$ is a transposed ``stochastic isometry'', in the sense that $p^t$ is an orthogonal isometry, with sum $1$ on rows.
\end{enumerate}
\end{definition}

Observe that, since the entries $p_{ij}$ of our various rectangular matrices are assumed to be self-adjoint, we have $p^*=p^t$. Thus the condition $pp^t=1$ reads $(p^t)^*p^t=1$, so the transposed matrix $q=p^t$ must indeed satisfy the isometry condition $q^*q=1$.

Observe also that the cubic condition on transposed orthogonal isometry $p$ is equivalent to the fact that the entries $x=p_{ij}$ satisfy the ``cubic'' condition $x^3=x$.

Note that we have surjective maps $C_+(G_n/G_k)\to C_\times(G_n/G_k)$, for any $G$ and any $k\leq n$. This follows indeed by comparing Proposition 4.1 and Definition 5.1.

As a fourth observation, the canonical map $C_+(G_n/G_0)\to C_\times(G_n/G_0)=C(G_n)$ is an isomorphism for $G=S^+,H^+$. For $G=O^+,B^+$ the situation is quite unclear, and finding better axioms in these cases is a question that we would like to raise here.

Finally, observe that in the case $G=O^+$ and $k=n-1$ we obtain the algebra of functions on the ``free sphere'', constructed and studied in \cite{bgo}. This will be actually our guiding example, the results obtained below being partly inspired by those in \cite{bgo}.

We will need the following key observation:

\begin{proposition}
For a transposed orthogonal isometry $p$, the following are equivalent:
\begin{enumerate}
\item $p$ is magic.

\item $p$ is cubic and stochastic.
\end{enumerate}
\end{proposition}

\begin{proof}
As already mentioned in section 4 above, at $k=n$ this result is well-known. In the general case the proof is similar, by using some basic $C^*$-algebra tricks:

$(1)\implies (2)$. Assume indeed that $p$ is magic. The transposed isometry condition $pp^t=1$ tells us that we have $\sum_jp_{ij}p_{kj}=\delta_{ik}$. At $i=k$ we get $\sum_jp_{ij}^2=1$, and since the elements $p_{ij}$ are projections, this condition becomes $\sum_jp_{ij}=1$. Thus $p$ is stochastic.

With this observation in hand, and since projections summing up to 1 must commute, we conclude that the elements $p_{ij}$ mutually commute on rows, so $p$ is cubic as well.

$(2)\implies (1)$. Assume that $p$ is cubic and stochastic. Since the elements $p_{i1},\ldots,p_{in}$ are self-adjoint, satisfy $xy=0$, and sum up to 1, they are projections, and we are done.
\end{proof}

We recall from section 1 that the homogeneous spaces of type $G/H$ are endowed with right coactions. Such a coaction is a morphism of $C^*$-algebras $\alpha:A\to A\otimes C(G)$ satisfying $(\alpha\otimes id)\alpha=(id\otimes\Delta)\alpha$, and such that $\alpha(A)(1\otimes C(G))$ is dense in $A\otimes C(G)$.

Given a trace $\varphi:A\to\mathbb C$, we denote by $A_{red}$ the quotient of $A$ by the null ideal of $\varphi$. Equivalently, $A_{red}$ is obtained by the GNS construction with respect to $\varphi$.

We recall also that the abelianized version of an algebra $A$ is the algebra $A_{class}$ obtained by quotienting $A$ by its commutator ideal.

\begin{theorem}
The algebras $C_+(G_n/G_k)$ and $C_\times(G_n/G_k)$ have the following properties:
\begin{enumerate}
\item They have coactions of $G_n$, given by $\alpha(p_{ij})=\sum_sp_{is}\otimes u_{sj}$.

\item They have unique $G_n$-invariant states, which are tracial.

\item Their reduced algebra versions are isomorphic.

\item Their abelianized versions are isomorphic.
\end{enumerate}
\end{theorem}

\begin{proof}
We follow the proof in \cite{bgo}, where this result was proved at $G=O^+$ and $k=n-1$. The only problems, requiring some new ideas, will appear in (4) at $G=S^+,H^+$.

(1) For $C_\times(G_n/G_k)$ this is clear, because this algebra is ``embeddable'' in the sense of \cite{pod}, and the coaction of $G_n$ is simply the restriction of the comultiplication map.

For the algebra $C_+(G_n/G_k)$, consider the following elements:
$$P_{ij}=\sum_{s=1}^np_{is}\otimes u_{sj}$$

We have to check that these elements satisfy the same relations as those in Definition 5.1, presenting the algebra $C_+(G_n/G_k)$, and the proof here goes as follows:

$O^+$ case. First, since $p_{ij},u_{ij}$ are self-adjoint, so is $P_{ij}$. Also, we have:
$$\sum_jP_{ij}P_{rj}=\sum_{jst}p_{is}p_{rt}\otimes u_{sj}u_{tj}=\sum_{st}p_{is}p_{rt}\otimes\delta_{st}=\sum_sp_{is}p_{rs}\otimes 1=\delta_{ir}$$

$H^+$ case. The condition $xy=0$ on rows is checked as follows ($j\neq r$):
$$P_{ij}P_{ir}=\sum_{st}p_{is}p_{it}\otimes u_{sj}u_{tr}=\sum_sp_{is}\otimes u_{sj}u_{sr}=0$$

$B^+$ case. The sum 1 condition on rows is checked as follows:
$$\sum_jP_{ij}=\sum_{js}p_{is}\otimes u_{sj}=\sum_sp_{is}\otimes 1=1$$

$S^+$ case. Since $P^t$ is cubic and stochastic, we just check the projection condition:
$$P_{ij}^2=\sum_{st}p_{is}p_{it}\otimes u_{sj}u_{tj}=\sum_sp_{is}\otimes u_{sj}=P_{ij}$$

Summmarizing, $P$ satisfies the same conditions as $p$, so we can define a morphism of $C^*$-algebras $\alpha:C_+(G_n/G_k)\to C_+(G_n/G_k)\otimes C(G_n)$ by $\alpha(p_{ij})=P_{ij}$. We have:
\begin{eqnarray*}
(\alpha\otimes id)\alpha (p_{ij})
&=&\sum_s\alpha(p_{is})\otimes u_{sj}=\sum_{st}p_{it}\otimes u_{ts}\otimes u_{sj}\\
(id\otimes\Delta)\alpha(p_{ij})
&=&\sum_tp_{it}\otimes\Delta(u_{ij})=\sum_{st}p_{it}\otimes u_{ts}\otimes u_{sj}
\end{eqnarray*}

Thus our map $\alpha$ is coassociative. The density conditions can be checked by using dense subalgebras generated by $p_{ij}$ and $u_{st}$, and we are done.

(2) For the existence part we can use the following composition, where the first two maps are the canonical ones, and the map on the right is the integration over $G_n$:
$$C_+(G_n/G_k)\to C_\times(G_n/G_k)\subset C(G_n)\to\mathbb C$$

Also, the uniqueness part is clear for the algebra $C_\times(G_n/G_k)$, as a particular case of the general properties of ``embeddable'' coactions, i.e. those coactions that can be realized as coactions on subalgebras of $C(G)$, via the restriction of the comultiplication.

For the uniqueness for $C_+(G_n/G_k)$, we use a method from \cite{bgo}. Let $h$ be the Haar state on $G_n$, and $\varphi$ be the $G_n$-invariant state constructed above. We claim that $\alpha$ is ergodic:
$$(id\otimes h)\alpha=\varphi(.)1$$

Indeed, let us recall from \cite{bsp} that the Haar state $h$ is given by the following ``Weingarten formula'', where $W_{sn}=G_{sn}^{-1}$, with $G_{sn}(\pi,\sigma)=n^{|\pi\vee\sigma|}$:
$$h(u_{i_1j_1}\ldots u_{i_sj_s})=\sum_{\pi,\sigma\in D(s)}\delta_\pi(i)\delta_\sigma(j)W_{sn}(\pi,\sigma)$$

Now, let us go back now to our claim. By linearity it is enough to check the above equality on a product of basic generators $p_{i_1j_1}\ldots p_{i_sj_s}$. The left term is as follows:
\begin{eqnarray*}
(id\otimes h)\alpha(p_{i_1j_1}\ldots p_{i_sj_s})
&=&\sum_{l_1\ldots l_s}p_{i_1l_1}\ldots p_{i_sl_s}h(u_{l_1j_1}\ldots u_{l_sj_s})\\
&=&\sum_{l_1\ldots l_s}p_{i_1l_1}\ldots p_{i_sl_s}\sum_{\pi,\sigma\in D(s)}\delta_\pi(l)\delta_\sigma(j)W_{sn}(\pi,\sigma)\\
&=&\sum_{\pi,\sigma\in D(s)}\delta_\sigma(j)W_{sn}(\pi,\sigma)\sum_{l_1\ldots l_s}\delta_\pi(l)p_{i_1l_1}\ldots p_{i_sl_s}
\end{eqnarray*}

Let us look now at the sum on the right. We have to sum the elements of type $p_{i_1l_1}\ldots p_{i_sl_s}$, over all multi-indices $l=(l_1,\ldots,l_s)$ which fit into our partition $\pi\in D(s)$. In the case of a one-block partition this sum is simply $\Sigma_lp_{i_1l}\ldots p_{i_sl}$, and we claim that:
$$\sum_lp_{i_1l}\ldots p_{i_sl}=\delta_\pi(i)$$

Indeed, by using the explicit description of the sets of diagrams $D(s)$ given in section 3 above, the proof of this formula goes as follows:

$O^+$ case. Here our one-block partition must be a semicircle, $\pi=\cap$, and the formula to be proved, namely $\sum_lp_{il}p_{jl}=\delta_{ij}$, follows from $pp^t=1$.

$S^+$ case. Here our one-block partition can be any $s$-block, $1_s\in P(s)$, and the formula to be proved, namely $\Sigma_lp_{i_1l}\ldots p_{i_sl}=\delta_{i_1,\ldots,i_s}$, follows from orthogonality on columns, and from the fact that the sum is 1 on rows.

$B^+$ case. Here our one-block partition must be a semicircle or a singleton. We are already done with the semicircle, and for the singleton the formula to be proved, namely $\sum_lp_{il}=1$, follows from the fact that the sum is 1 on rows.

$H^+$ case. Here our one-block partition must have an even number of legs, $s=2r$, and due to the cubic condition the formula to be proved reduces to $\sum_lp_{il}^{2r}=1$. But since $p_{il}^{2r}=p_{il}^2$, independently on $r$, the result follows from the orthogonality on rows.

In the general case now, since $\pi$ noncrossing, the computations over the blocks will not interfere, and we will obtain the same result, namely:
$$\sum_lp_{i_1l}\ldots p_{i_sl}=\delta_\pi(i)$$

Now by plugging this formula into the computation that we have started, we get:
\begin{eqnarray*}
(id\otimes h)\alpha(p_{i_1j_1}\ldots p_{i_sj_s})
&=&\sum_{\pi,\sigma\in D(s)}\delta_\pi(i)\delta_\sigma(j)W_{sn}(\pi,\sigma)\\
&=&h(u_{i_1j_1}\ldots u_{i_sj_s})\\
&=&\varphi(p_{i_1j_1}\ldots p_{i_sj_s})
\end{eqnarray*}

This finishes the proof of our claim. So, let us get back now to the original question. Let $\tau:C_+(G_n/G_k)\to\mathbb C$ be a linear form as in the statement. We have:
$$\tau (id\otimes h)\alpha(x)
=(\tau\otimes h)\alpha(x)
=h(\tau\otimes id)\alpha(x)
=h(\tau(x)1)
=\tau(x)$$

On the other hand, according to our above claim, we have as well:
$$\tau(id\otimes h)\alpha(x)=\tau(\varphi(x)1)=\varphi(x)$$

Thus we get $\tau=\varphi$, which finishes the proof of the uniqueness assertion.

(3) This follows from the uniqueness assertions in (2), and from some standard facts regarding the reduced versions with respect to Haar states, from \cite{wo3}.

(4) We denote by $G^-$ the classical version of $G$, given by $G^-=O,S,H,B$ in the cases $G=O^+,S^+,H^+,B^+$. We have surjective morphisms of algebras, as follows:
$$C_+(G_n/G_k)\to C_\times(G_k/G_k)\to C_\times(G_n^-/G_k^-)=C(G_n^-/G_k^-)$$

Thus at the level of abelianized versions, we have surjective morphisms as follows:
$$C_+(G_n/G_k)_{comm}\to C_\times(G_n/G_k)_{comm}\to C(G_n^-/G_k^-)$$

In order to prove our claim, namely that the first surjective morphism is an isomorphism, it is enough to prove that the above composition is an isomorphism.

Let $r=n-k$, and denote by $A_{n,r}$ the algebra on the left. This is by definition the algebra generated by the entries of a transposed $n\times r$ isometry, whose entries commute, and which is respectively orthogonal, magic, cubic, bistochastic. We have a surjective morphism $A_{n,r}\to C(G_n^-/G_k^-)$, and we must prove that this is an isomorphism.

$S^+$ case. Since $\#(S_n/S_k)=n!/k!$, it is enough to prove that $\dim(A_{n,r})=n!/k!$. Let $p_{ij}$ be the standard generators of $A_{n,r}$. By using the Gelfand theorem, we can write $p_{ij}=\chi(X_{ij})$, where $X_{ij}\subset X$ are certain subets of a given set $X$. Now at the level of sets the magic isometry condition on $(p_{ij})$ tells us that the matrix of sets $(X_{ij})$ has the property that its entries are disjoint on columns, and form partitions of $X$ on rows.

So, let us try to understand this property for $n$ fixed, and $r=1,2,3,\ldots$

-- At $r=1$ we simply have a partition $X=X_1\sqcup\ldots\sqcup X_n$. So, the universal model can be any such partition, with $X_i\neq 0$ for any $i$.

-- At $r=2$ the universal model is best described as follows: $X$ is the $n\times n$ square in $\mathbb R^2$ (regarded as a union of $n^2$ unit tiles) minus the diagonal, the sets $X_{1i}$ are the disjoint unions on rows, and the sets $X_{2i}$ are the disjoint unions on columns.

-- At $r\geq 3$, the universal solution is similar: we can take $X$ to be the $n^r$ cube in $\mathbb R^r$, with all tiles having pairs of equal coordinates removed, and say that the sets $X_{si}$ for $s$ fixed are the various ``slices'' of $X$ in the direction of the $s$-th coordinate of $\mathbb R^r$.

Summarizing, the above discussion tells us that $\dim(A_{n,r})$ equals the number of tiles in the above set $X\subset\mathbb R^r$. But these tiles correspond by definition to the various $r$-tuples $(i_1,\ldots,i_r)\in\{1,\ldots,n\}^r$ with all $i_k$ different, and since there are exactly $n!/k!$ such $r$-tuples, we obtain $\dim(A_{n,r})=n!/k!$, and we are done.

$H^+$ case. We can use here the same method as for $S_n^+$. This time the functions $p_{ij}$ take values in $\{-1,0,1\}$, and the algebra generated by their squares $p_{ij}^2$ coincides with the one computed above for $S_n^+$, having dimension $n!/k!$. Now by taking into account the $n-k$ possible signs we obtain $\dim(A_{n,r})\leq 2^{n-k}n!/k!=\#(H_n/H_k)$, and we are done.

$O^+$ case. We can use the same method, namely a straightforward application of the Gelfand theorem. However, instead of performing a dimension count, which is no longer possible, we have to complete here any transposed $n\times r$ isometry whose entries commute to a $n\times n$ orthogonal matrix. But this is the same as completing a system of $r$ orthogonal norm 1 vectors in $\mathbb R^n$ into an orthonormal basis of $\mathbb R^n$, which is of course possible.

$B^+$ case. Since we have a surjective map $C(O_n^+)\to C(B_n^+)$, we obtain a surjective map $C_+(O_n^+/O_k^+)\to A_{n,r}$, and hence surjective maps as follows:
$$C(O_k/O_k)\to A_{n,r}\to C(B_n/B_k)$$

Now since this composition is the canonical map $C(O_k/O_k)\to C(B_n/B_k)$, by looking at the column vector $\xi=(1,\ldots,1)^t$, which is fixed by the stochastic matrices, we conclude that the map on the right is an isomorphism, and we are done.
\end{proof}

The big problem that we have is that of understanding the kernel of the surjective map $C_+(G_n/G_k)\to C_\times(G_n/G_k)$. As already mentioned after Definition 5.1, in the cases $G=O,B$ the kernel is already present at $k=0$, where this map is proper.

In the cases $G=S,H$, however, we do not have any properness problems at $k=0$, and we believe that the axioms in Definition 5.1 are basically the good ones. However, the question of computing the kernel of $C_+(G_n/G_k)\to C_\times(G_n/G_k)$, and perhaps of improving the definition of $C_+(G_n/G_k)$, stands as an open problem, for generic values of $k$.

We would like to end with a few details in the case of the algebra $C_+(S_n^+/S_k^+)$, which seems to be a quite interesting object. First, we have the following result:

\begin{proposition}
The quotient map $C_+(S_n^+/S_k^+)\to C_\times(S_n^+/S_k^+)$ is an isomorphism at $k\in\{0,1,n-1,n\}$.
\end{proposition}

\begin{proof}
Indeed, at $k=0,1$, at $k=n-1$ and at $k=n$ the quotient map in the statement is respectively the identity of the algebras $C(S_n^+)$, $\mathbb C^n$, and $\mathbb C$.
\end{proof}

The key problem is to decide what happens at $2\leq k\leq n-2$. The first problem appears at $n=4,k=2$, where we do not know if the above surjection is proper or not. The matrix model for $C(S_4^+)$ found in \cite{bco} can be probably used here, but the subject is quite tricky and technical, and we do not have further results on this question.

We do not know either whether the properness property at $n=4,k=2$ would imply the properness property at any $2\leq k\leq n-2$.

In addition, with $r=n-k$ we believe the isomorphism question to be equivalent to the fact that ``any $n\times r$ magic isometry can be completed to a $n\times n$ magic unitary''. However, we do not know how to prove or disprove this statement, nor how to prove that this statement is indeed equivalent to the isomorphism question for the above maps.

Observe that the usual symmetric group $S_n$ cannot help here, because at the level of abelianized versions we always have isomorphism, as explained in Theorem 5.3 above.

\section{Concluding remarks}

We have seen in this paper that the study of quantum homogeneous spaces $G/H$, where $G\subset U_n^+$ is quantum subgroup and $H=G\cap U_k^+$ with $k\leq n$, leads to some interesting combinatorics, and to a number of concrete results. We have the following questions:

(1) Is it possible to generalize Theorem 2.5, by using some suitable discrete quantum group notions of  normality and normal closure? In principle the normality condition $\Lambda\triangleleft\Gamma$ can be formulated in functional analytic terms as $C_0(\Lambda\backslash\Gamma)=C_0(\Gamma/\Lambda)$, cf. \cite{wa3}, but we do not know if this can help in connection with our problem.

(2) In the easy case, when is the map $C_+(G/H)\to C_\times(G/H)$ proper? Note that an answer for $G=O_n^+$ and $k=n-1$ would clarify the axiomatization of the free spheres in \cite{bgo}, and an answer for $G=S_n^+$ would probably bring some advances on the (quite untractable) free hypergeometric laws, recently introduced in \cite{bbc}.

(3) Does the situation $C_+(G/H)\to C_\times(G/H)\subset C(G/H)$, as described in Theorem 5.3, extend beyond the easy case? One may wonder for a ``category of algebras of type $C_\alpha(G/H)$'', with a maximal and minimal object, as in Theorem 5.3. This problem is quite abstract, and we do not have further results here.

(4) What is the precise definition, and the general theory, of double coset spaces of type $H\backslash G/H$? An answer here for $G=S_n^+$, $H=S_k^+$ would be very interesting in connection with the free hypergeometric laws, introduced in \cite{bbc}. The present paper is a first step towards such a theory, but we do not have further results.

\end{document}